\documentclass[11pt]{amsart}
\usepackage[utf8]{inputenc}
\usepackage{todonotes}
\numberwithin{equation}{section}
\newtheorem{theorem}{Theorem}[section]
\newtheorem{lemma}[theorem]{Lemma}
\theoremstyle{remark}
\newtheorem{example}[theorem]{Example}

\DeclareMathOperator{\dist}{dist}
\title{Uniform convergence of Green's functions}

\author{Sergei Kalmykov}
\address{School of mathematical sciences, Shanghai Jiao Tong University, 800 Dongchuan RD, Shanghai 200240, China}
\email{sergeykalmykov@inbox.ru}
\thanks{First author supported by NSFC grant 11650110426.}

\author{Leonid V. Kovalev}
\address{215 Carnegie, Mathematics Department, Syracuse University, Syracuse, NY 13244, USA}
\email{lvkovale@syr.edu}
\thanks{Second author supported by the National Science Foundation grant DMS-1362453.}


\begin{document}

\begin{abstract}
Given a sequence of regular planar domains converging in the sense of kernel,  we prove that the corresponding Green's functions converge uniformly on the complex sphere, provided the limit domain is also regular, and the connectivity is uniformly bounded.
\end{abstract}

\maketitle

\section{Introduction}

Convergence theorems for holomorphic maps, harmonic functions, and Green's functions usually concern their uniform convergence on compact subsets of a domain (e.g., ~\cite{Duren, Pommerenke}). However, their uniform convergence on all of the domain is also of interest. A classical result on the uniform convergence of normalized conformal maps $f_n\colon \mathbb D\to\mathbb C$ is a 1922 theorem of Rad\'o~\cite{Rado}, which related it to the Frech\'et distance between boundary curves. Rad\'o theorem has recently become a useful tool in the  theory of Schramm-Loewner evolution~\cite{CamiaNewman, CamiaNewman2}; it was also extended to quasiconformal maps~\cite{Palka}.

In this note we study the uniform convergence of Green's functions of domains on the complex sphere and in $\mathbb R^n$. Under the Dirichlet boundary condition, Green's function of a regular domain can be continuously extended by $0$ outside of the domain, which makes it possible to discuss uniform convergence on the entire space. We will prove that uniform convergence holds under the kernel convergence of planar domains, provided their connectivity is uniformly bounded (Theorem~\ref{multiplyconnected}). In higher dimensions, uniform convergence may fail even for domains homeomorphic to a ball (Example~\ref{examplespace}).

\section{Preliminaries}

All domains we consider are assumed to be regular for the Dirichlet boundary problem. These can be characterized as the domains $\Omega\subset \overline{\mathbb C}$ such that for every $w\in \Omega$ there exists \textit{Green's function} $g_\Omega(\cdot, w)$ such that
\begin{itemize}
\item $g_\Omega(\cdot, w)$ is continuous on $\overline{\mathbb C}\setminus \{w\}$ and is zero on $\overline{\mathbb C}\setminus \Omega$;
\item $g_\Omega(\cdot, w)$ is harmonic in $\Omega\setminus \{w\}$;
\item $g_\Omega(z,w) = -\log|z-w| + O(1)$ as $z\to w$ if $w$ is finite;
\item $g_\Omega(z,w)=\log|z|+O(1)$, as $z\rightarrow w$ if  $w=\infty$.
\end{itemize}

As for higher dimensions, the {\it Green's function} $g(\cdot,w)$, of a domain $\Omega\subset \mathbb{R}^n$,~$n>2$ with the pole at $w\in \Omega$ is defined by the following conditions.
\begin{itemize}
\item $g_\Omega(\cdot, w)$ is continuous on $\mathbb{R}^n \setminus \{w\}$ and is zero on ${\mathbb R}^n\setminus \Omega$;
\item $g_\Omega(\cdot, w)$ is harmonic in $\Omega\setminus \{w\}$;
\item $g_\Omega(z,w) = |z-w|^{2-n} + O(1)$ as $z\to w$.
\end{itemize}
We write $g=g_\Omega$ and $g_n=g_{\Omega_n}$.

Recall that monotone convergence of domains (that is, $\Omega_n\subset \Omega_{n+1}$ or $\Omega_n\supset \Omega_{n+1}$ for all $n$) implies pointwise convergence of their Green's functions; moreover this convergence is uniform on compact subsets of the limit domain~\cite[p. 94]{MR1814344}. Such convergence properties are relevant for the studies of the asymptotic behavior of capacities of planar sets and condensers~\cite{Dubinin}. A concept more general than  monotone convergence is kernel convergence, which we define next.

Let $w\in\mathbb{C}$ be given and let $\Omega_n\subset\overline{\mathbb C}$ be domains such that $w$ is an interior point of $\bigcap_{n=1}^\infty \Omega_n$.
Following~\cite[p. 13]{Pommerenke}, we say that
\[
\Omega_n \rightarrow \Omega \ \ \text{as} \ \ n\rightarrow \infty \ \ \text{with respect to } w
\]
in the sense of {\it kernel convergence} if
\begin{itemize}
\item $\Omega$ is a domain such that $w\in \Omega$ and some neighborhood of every $z\in \Omega$ lies in $\Omega_n$ for large $n$;
\item for $z\in \partial \Omega$ there exist $z_n\in \partial \Omega_n$ such that $z_n\rightarrow z$ as $n\rightarrow \infty$.
\end{itemize}
An equivalent definition is found in ~\cite[p.~77]{Duren} and \cite[p. 54]{MR0247039}. According to it, $\Omega$ is the kernel of $\{\Omega_n\}$ if it is the maximal domain containing $w$ such that every compact subset of $\Omega$ belongs to all but finitely many of the domains $\Omega_n$. The convergence to $\Omega$ in the sense of kernel requires that every subsequence of $\{\Omega_n\}$ also has the same kernel $\Omega$.

The definition of kernel convergence is also applicable to the convergence of domains in $\mathbb R^n$ for $n>2$.

\begin{lemma}\label{oneside} If  $\Omega_n\to \Omega$ in the sense of kernel, then  $\sup_{z\ne w} (g(z,w) - g_n(z,w)) \to 0$ for every $w\in\Omega$.
\end{lemma}

\begin{proof} Let $M_n=\sup_{z\ne w} (g(z,w) - g_n(z,w))$. Since $\Omega_n$ is a regular domain, we have $g_n(z,w)=0$ for $z\in \partial \Omega_n$. Hence $M_n\ge 0$.

Given $\epsilon>0$, let $U=\{z\in\Omega\colon g(z,w)>  \epsilon \}$. Since $\overline{U}$ is a compact subset of $\Omega$, there exists $N$ such that $\overline{U}\subset \Omega_n$ for all $n\ge N$. Applying the maximum principle to $g-g_n -\epsilon$ on the set $U$ yields $g-g_n\le \epsilon$ on $U$, since $g=\epsilon$ on $\partial U$. On the complement of $U$, the inequality $g-g_n\le \epsilon$ holds by virtue of $g \le \epsilon$. This proves $M_n\le \epsilon$.
\end{proof}

Finally, let us introduce the notation $D(a,r)$ for the open disk of radius $r$ and center $a$, and $C(a,r) = \partial D(a,r)$. As a special case, $\mathbb D=D(0,1)$.

\section{Simply connected domains}

In this section we consider the important special case of simply-connected domains in $\mathbb C$. The proof of this case is different from the proof of Theorem~\ref{multiplyconnected} concerning the multiply connected domains.

\begin{theorem}\label{simplyconnected} Suppose that simply-connected regular domains $\Omega_n\subset\mathbb C$
converge in the sense of kernel
to a regular simply-connected domain $\Omega$.
Then for every $w\in\Omega$ Green's functions $g_{n}(\cdot, w)$ converge to $g(\cdot, w)$ uniformly on $\mathbb C\setminus \{w\}$.
 \end{theorem}

\begin{proof} Suppose the uniform convergence fails. Using Lemma~\ref{oneside} and passing to a subsequence, we may assume there exist $\epsilon>0$ such that $\sup_{\mathbb C\setminus \{w\}}(g_n-g)\ge \epsilon$ for all $n$. The function $g_n-g$ is bounded and harmonic on the set $U = \mathbb C\setminus (\partial \Omega_n\cup\partial \Omega )$  after the removable singularity at $w$ is eliminated. Since $g_n-g \le 0$ on $\partial \Omega_n$, it follows from the maximum principle that $g_n-g$ attains its maximum on $\partial\Omega$.

Pick $z_n\in\partial\Omega$ such that $g_n(z_n)\ge  \epsilon$.

We claim that the sequence $\{z_n\}$ is bounded. Indeed, the distances $\dist(w,\partial\Omega_n)$ must be bounded by some constant $d$, for otherwise the limit domain $\Omega$ would be $\mathbb C$, which does not have Green's function. Let $h$ be Green's function of the domain $\mathbb C\setminus (-\infty, -d]$ with the pole at $0$. A theorem of Krzy\.z~\cite{Krzyz} (see also~\cite{Baernstein}) asserts that Green's function increases under circular symmetrization of the domain, which implies
$g_n(z_n) \le h(|z_n-w|)$. Since $h(t)\to 0$ as $t\to\infty$, the fact that $g_n(z_n)\ge \epsilon$  implies that $\{z_n\}$ is a bounded sequence.

By compactness, passing to a subsequence, we may assume that $z_n\to z\in\mathbb C$. Let $\phi_n\colon \mathbb D\to \Omega_n$ be a conformal map such that $\phi_n(0)=w$. For every $\zeta\in \mathbb D$ the Koebe $1/4$-theorem and the distortion theorem yield
\begin{equation}\label{Koebe14}
\dist(\phi_n(\zeta), \partial\Omega_n) \ge
\frac14 |\phi_n'(\zeta)|(1-|\zeta|) \ge \frac{(1-|\zeta|)^2}{4(1+|\zeta|)^3}
\ge \frac{(1-|\zeta|)^2}{32}
\end{equation}
Note also that $g_n(\phi_n(\zeta)) = -\log|\zeta| $.

 When $1-|\zeta|\le 1/2$, we have $-\log|\zeta |\le 2(1-|\zeta|)$.
If a point $z=\phi_n(\zeta)$ satisfies $\dist(z, \partial\Omega_n) \le 1/128$, then~\eqref{Koebe14} implies $1-|\zeta|\le \sqrt{32 \dist(z,\partial\Omega_n)}\le 1/2$, hence
\[
g_n(z) = -\log|\zeta| \le 2(1-|\zeta|) \le \sqrt{128 \dist(z, \partial\Omega_n)}
\]

Therefore, if $z_n\to z$ is a sequence such that $g_n(z)\ge \epsilon$ for all $n$, then $\dist(z,\partial \Omega_n) \ge c>0$ for all sufficiently large $n$. This contradicts the definition of kernel convergence, since  $z\in\partial\Omega$.
\end{proof}

\section{Multiply connected domains}

Here we consider multiply connected domains on the Riemann sphere $\overline{\mathbb{C}}$, subject to a uniform upper bound on their connectivity.

\begin{theorem}\label{multiplyconnected} Suppose that regular domains $\Omega_n\subset \overline{\mathbb C}$ converge in the sense of kernel to a regular domain $\Omega$. If there exists a constant $N$ such that each domain $\Omega_n$ has at most $N$ boundary components, then for every $w\in\Omega$ Green's functions $g_{n}(\cdot, w)$ converge to $g(\cdot, w)$ uniformly on $\mathbb C\setminus \{w\}$.
\end{theorem}

\begin{lemma}\label{Greenbound} For every $\epsilon>0$ there exists $\delta \in (0,1)$ with the following property. If $\Omega\subset \overline{\mathbb C}$ is a domain containing distinct finite points $z, w$, and there is $\rho\le |z-w|/2$ such that some connected component of $\mathbb C\setminus \Omega$ meets both $C(z,\rho)$ and $C(z,\delta \rho)$, then $g_\Omega(z,w)<\epsilon$.
\end{lemma}

\begin{proof} This is a consequence of Theorem 2.2 in~\cite{Totik}. Indeed, let $\Gamma$ be a component of $\mathbb C\setminus \Omega$ mentioned in the statement.
After applying a M\"obius transformation that sends $z$ to $0$ and $w$ to $\infty$ and rescaling the resulting domain, $\Gamma$ is transformed to a compact connected set $E$ that meets both the unit circle $C(0,1)$ and the circle $C(0, r)$ with $r\sim \delta$. In particular, we have $r\le 1/2$ when $\delta$ is small enough. In this setting, Theorem 2.2~\cite{Totik} implies that $g_{\mathbb C\setminus E}(0,\infty) \le C\sqrt{r}$ with an absolute constant $C$. This proves the lemma, since $r$ can be made arbitrarily small by choosing $\delta$ small.
\end{proof}

\begin{proof}[Proof of Theorem~\ref{multiplyconnected}]
Suppose that the uniform convergence fails. As in the proof of Theorem~\ref{simplyconnected}, we have $\epsilon>0$ and a sequence $z_n\in\partial\Omega$ such that $g_n(z_n)\ge  \epsilon$. Passing to a subsequence, we get $z_n\to z_0\in \partial \Omega$, the convergence being with respect to the spherical metric on $\overline{\mathbb C}$. Applying a M\"obius transformation, we can make sure that $z_0$ and $w$ are finite points in $\mathbb C$.

Let $d=|z_0-w|$. Pick a small $\lambda\in (0,1)$ and consider the concentric annuli
\[
A_j = \{z \colon d \lambda^{j+2} < |z-z_0|< d\lambda^{j}\}, \quad j\in\mathbb N,
\]
as well as their counterparts centered at $z_n$,
\[
A_j(n) = \{z \colon d \lambda^{j+2} < |z-z_n|< d\lambda^{j}\}, \quad j\in\mathbb N.
\]
Let $J$ be the set of indices $j\in\mathbb N$ such that $A_j$ is not contained in $\Omega$.

\textbf{Claim 1}: The set $J$ has at most $3N$ elements, where $N$ is as in the statement of the theorem.

Suppose to the contrary that $J$ contains a subset $J'$ of cardinality $3N+1$. By Lemma~\ref{Greenbound} no connected component of $\Omega_n^c$ meets more than one of the circles $\{z \colon |z-z_n| =  d\lambda^{j}\}$,
provided $\lambda$ is small enough. It follows that a component of $\Omega_n^c$ can intersect at most three of the annuli $A_j(n)$. Therefore, for each $n$ there exists $j_n\in J'$ such that  $A_{j_n}(n)\subset \Omega_n$. Since $J'$ is finite, there exists a subsequence of the domains $\Omega_n$ for which $j_n$ is the same index $j'\in J'$.

We claim $A_{j'}\subset \Omega$. Indeed, otherwise there is a compact subset $K\subset A_{j'}$ which is not contained in $\Omega$. But  $K\subset A_{j'}(n)$ for all sufficiently large $n$, which implies $K\subset \Omega_n$ for all such $n$. By the definition of kernel convergence, $K\subset \Omega$ and thus $A_{j'}\subset\Omega$. Since the containment $A_{j'}\subset\Omega$ contradicts  the definition of $J$, Claim~1 is proved.

Since $J$ is finite, $z_0$ is an isolated boundary point of $\Omega$, a contradiction to the regularity of domain $\Omega$.
\end{proof}

\section{Counterexamples}

The first example is standard; it demonstrates that without assuming the regularity of limit domain one cannot hope for uniform convergence even if the domains are nested.

\begin{example}\label{nonregularlimit}
The regular domains $\Omega_n= \{z : 1/n < |z|< 1\}$ converge in the  sense of kernel to the punctured unit disk $\Omega = \mathbb D\setminus\{0\}$. Since Green's function of $\Omega$ is the same as the one of $\mathbb D$, the convergence $g_n(z, 1/2)\to g(z,1/2)$ fails to be uniform in $\Omega$.
\end{example}

Our second example shows that a uniform bound on the number of boundary components in Theorem~\ref{multiplyconnected} cannot be omitted.

\begin{example}\label{exampleplane}
There exists a sequence of regular domains $\Omega_n\subset\mathbb C$ that converge in the sense of kernel to the unit disk $\mathbb D$ but their Green functions do not converge to $g_{\mathbb D}$ even pointwise.
\end{example}

\begin{proof}
To construct $\Omega_n$, let $A_n$ be a finite $(1/n)$-net in the set $D(0,2) \setminus \mathbb D$. The domain $D(0,2)\setminus A_n$ is not regular, and its Green's function is the same as of $D(0,2)$. For sufficiently small $r>0$ the closed $r$-neighborhood of $A_n$ consists of $n$ disjoint disks contained in $D(0,2)\setminus \overline{\mathbb D}$. Removing these disks from $D(0,2)$ we obtain a regular domain $\Omega(n,r)$.

As $r\to 0$, Green's function of $\Omega(n,r)$ converges pointwise to Green's function of $D(0,2)$. Choose $r$ small enough so that
\[g_{\Omega(n,r)}(1/2, 0) \ge g_{D(0,2)}(1/2, 0) - \frac1n\]
and let this domain be $\Omega_n$.

By construction, $\Omega_n\to\mathbb D$ in the sense of kernel convergence. However,
\[g_{\Omega_n}(1/2, 0) \to g_{D(0,2)}(1/2, 0) > g_{\mathbb D}(1/2,0)  \]
as claimed.
\end{proof}

The number of boundary components of $\Omega_n$ in the above example grows indefinitely, as it must according to Theorem~\ref{multiplyconnected}. However, in  higher dimensions such examples can be constructed with domains homeomorphic to a ball.

\begin{example}\label{examplespace}
For $d\ge 3$, there exists a sequence of regular domains $\Omega_n\subset\mathbb R^d$ that converge in the sense of kernel to the unit ball $\mathbb B$ but their Green functions do not converge to $g_{\mathbb B}$ even pointwise. Furthermore, each $\Omega_n$ can be chosen to be homeomorphic to $\mathbb B$.
\end{example}

\begin{proof} Let $A_n$ be a finite $(1/n)$-net in the set $B(0,2) \setminus \mathbb B$. Pick a simple $C^\infty$-smooth curve $\Gamma$ that passes through every point of $A_n$ and is contained in $B(0,2) \setminus \mathbb B$ except for one endpoint, which lies on $\partial B(0,2)$. Removing $\Gamma$ from $B(0,2)$ does not affect Green's function since the curve has codimension at least $2$ and is therefore a polar set (Theorems 5.14 and 5.18 in ~\cite{HaymanKennedy}).

For sufficiently small $r>0$ the closed $r$-neighborhood of $\Gamma$ is a tubular neighborhood with no self-intersections. Removing it from $B(0,2)$ disks we obtain a regular domain $\Omega(n,r)$ that is homeomorphic to $B(0,2)$. As in Example~\ref{exampleplane}, $r$ can be made small enough so that as $n\to\infty$, Green's function converges to $g_{B(0,2)}$ on compact subsets of $\mathbb B$, rather than to $g_{\mathbb B}$. In the sense of kernel, the domains converge to $\mathbb B$, however.
\end{proof}

\bibliographystyle{amsplain}
\bibliography{amsrefs}

\end{document}